\documentclass[12pt,draft]{article}
\usepackage{authblk}
\usepackage{hyperref}
\usepackage[left=2.4cm,right=1cm]{geometry}

\usepackage{amsthm}
\usepackage{amsmath}
\usepackage{amssymb}
\usepackage{tikz-cd}
\usepackage{graphicx} 
\usetikzlibrary{cd}
\usepackage{amscd}
\usepackage[all]{xy}
\usepackage{biblatex}
\usepackage{bussproofs}
\usepackage{thmtools}

\declaretheorem[name=Theorem,numberwithin=section]{theorem}
\declaretheorem[name=Proposition,sibling=theorem]{proposition}
\declaretheorem[name=Lemma,sibling=theorem]{lemma}
\declaretheorem[name=Corollary,sibling=theorem]{corollary}
\declaretheorem[name=Example,numberwithin=section]{example}
\declaretheorem[name=Remark,sibling=example]{remark}

\newcommand{\lra}{\leftrightarrows}

\newcommand{\up}{{\uparrow}}

\newcommand{\ca}[1]{\mathcal{#1}}
\newcommand{\bd}[1]{\mathbf{#1}}
\newcommand{\mf}{\mathsf}
\newcommand{\mi}{\mathit}
\newcommand{\mk}{\mathcal}

\newcommand{\se}{\subseteq}
\newcommand{\sm}{\setminus}

\newcommand{\we}{\wedge}
\newcommand{\ve}{\vee}
\newcommand{\bwe}{\bigwedge}
\newcommand{\bve}{\bigvee}
\newcommand{\bca}{\bigcap}

\newcommand{\opp}[1]{#1^{op}}

\newcommand{\op}{\mathfrak{o}}
\newcommand{\cl}{\mathfrak{c}}
\newcommand{\bl}{\mathfrak{b}}
\newcommand{\sll}{\mf{S}(L)}
\newcommand{\So}{\mf{S}_{\op}}
\newcommand{\Sc}{\mf{S}_{\cl}}
\newcommand{\Sb}{\mf{S}_b}
\newcommand{\See}{\mf{S}_{\ca{E}}}

\newcommand{\fe}{\mathsf{Filt}_{\mathcal{E}}}
\newcommand{\fse}{\mathsf{Filt}_{\mathcal{SE}}}

\newcommand{\Om}{\Omega}
\newcommand{\pt}{\mathsf{pt}}

\newcommand{\kF}{\mk{F}}
\newcommand{\kD}{\mk{D}}
\newcommand{\vF}{\vee^{\mk{F}}}
\newcommand{\wF}{\wedge^{\mk{F}}}
\newcommand{\bwf}{\bwe^{\mk{F}}}
\newcommand{\bvf}{\bve^{\mk{F}}}

\title{Strictly zero-dimensional biframes and Raney extensions}
\author{Anna Laura Suarez\thanks{University of the Western Cape, Robert Sobukwe Road, Bellville 7535, Cape Town, South Africa}}
\date{}

\begin{document}

\maketitle
\begin{abstract}
Raney extensions and strictly zero-dimensional biframes both faithfully extend the dual of the category of $T_0$ spaces. We use tools from pointfree topology to look at the connection between the two. Raney extensions may be equivalently described as pairs $(L,\kF)$ where $L$ is a frame and $\kF\se \So(L)$ a subcolocale containing all opens. Here, $\So(L)$ is the collection of all fitted sublocales of $L$. Similarly, a strictly zero-dimensional biframe is a pair $(L,\kD)$ where $\kD\se \sll$ is a codense subcolocale. We show that there is an adjunction between certain subcolocales of $\So(L)$ and codense subcolocales of $\sll$. We show that the adjunction maximally restricts to an order-isomorphism between what we call the \emph{proper} subcolocales of $\So(L)$ and the \emph{essential} codense subcolocales. As an application of our main result, we establish a bijection between proper Raney extensions and the strictly zero-dimensional biframes $(L_1,L_2,L)$ such that $L$ is an essential extension of $L_2$ in $\bd{Frm}$. We show that this correspondence cannot be made functorial in the obvious way, as a frame morphism $f:L\to M$ may lift to a map $f:(L,\kF)\to (L,\mk{G})$ of Raney extensions without lifting to a map between the associated strictly zero-dimensional biframes.
\end{abstract}

\tableofcontents

\section*{Introduction}

The usual approach in pointfree topology is to consider the adjunction $\Om:\bd{Top}\lra \bd{Frm}^{op}:\pt$ between frames and spaces, and to regard frames as pointfree spaces in virtue of this. This is the classical approach found, for example, in \cite{johnstone82}, \cite{PP12}, \cite{picado21}. The fixpoints on the $\bd{Top}$ side are the sober spaces, so we may view $\bd{Frm}$ as a faithful extension of the opposite of the category $\bd{Sob}$ of sober spaces. An alternative approach is \emph{$T_D$-duality}, developed in \cite{banaschewskitd}. This is based on the $T_D$-axiom, introduced in \cite{Aull62}. The full subcategory of $\bd{Top}$ consisting of the $T_D$-spaces is related via a similar adjunction to a wide subcategory of $\bd{Frm}$, the category of frames with \emph{D-morphisms}, called $\bd{Frm_D}$. The category $\bd{Frm_D}$ is thus shown to faithfully extend the opposite of $\bd{Top_D}$. Sobriety and the $T_D$ property are incomparable. This is why the language of frames, under both translations, is not expressive enough to capture certain concepts and constructions. For example, in $\bd{Frm}$ we do not have a notion of sobrification. Similarly, in $\bd{Frm_D}$, we do not have a notion of $T_D$-coreflection (which exists for spaces). Refining the language by extending the category $\bd{Frm}$ enables us to capture both things in the same category. For example, in \cite{suarez25}, the category $\bd{Frm}$ is extended to the category $\bd{Raney}$ of \emph{Raney extensions}, which also faithfully extends the dual of the category $\bd{Top_0}$ of all $T_0$-spaces. In this category, we have both the notion of \emph{sober coreflection} (Section 6.5), the dual notion of sobrification, and \emph{$T_D$-reflection} (Section 6.6), dual of the $T_D$-coreflection.

\begin{itemize}
\item Raney extensions are inspired by the work of Raney on completely distributive lattices, see for example \cite{raney52}. Raney duality, as illustrated in \cite{bezhanishvili20}, is the result that the dual of $\bd{Top_0}$ is equivalent to the category of completely distributive lattices $\ca{U}(X)$ consisting of upper sets of some poset $X$, equipped with an interior operator. A \emph{Raney extension} is a pair $(L,C)$ where $C$ is a coframe and $L\se C$ is a frame with the inherited order, such that the subset inclusion preserves all joins and strongly exact meets. The Raney extension corresponding to a space $X$ is the pair $(\Om(X),\ca{U}(X))$, where $\Om(X)$ is its frame of opens and $\ca{U}(X)$ is the collection of upper sets in its specialization order. The category $\bd{Raney}$ of Raney extensions faithfully extends $\bd{Top_0}^{op}$.

\item Another approach to refining the language of frames is that of \emph{McKinsey-Tarski algebras}, as introduced in \cite{bezhanishvili23}. This is based on work by McKinsey and Tarski, who in \cite{MT1944} studied topological spaces in terms of the closure operator they induce on their powerset. McKinsey-Tarski algebras make this approach pointfree, by considering as objects complete Boolean algebras, not necessarily atomic, with interior operators. The category of MT-algebras faithfully extends all of $\bd{Top}^{op}$.

\item A third approach to faithfully extending the dual of $\bd{Top_0}$ is that of \emph{strictly zero-dimensional biframes}, as shown in \cite{manuell15}. Although in \cite{manuell15} the pointfree description of $T_0$ spaces is not the focus, it is indeed observed (end of Section 4) that there is a dual adjunction between spaces and strictly zero-dimensional biframes, whose fixpoints are the $T_0$ spaces. 
\end{itemize}
It is then interesting to look at how the three categories interact, and go towards a more unified theory of pointfree $T_0$ spaces. The connection between Raney extensions and MT-algebras is looked at in \cite{bezhanishvili2025mckinseytarskialgebrasraneyextensions}. A Raney extension $(L,C)$ is \emph{proper} if the joins of $L$ distribute over all binary meets in $C$. In \cite{bezhanishvili2025mckinseytarskialgebrasraneyextensions} the connection between Raney extensions and MT-algebras is explored, and it is shown that the category of proper Raney extensions is equivalent to the category $\bd{MT_0}$ of $T_0$ MT-algebras equipped with a notion of morphism based on a proximity-like relation. 

With this paper, we add another part to the big picture, connecting explicitly the categories of Raney extensions and that of strictly zero-dimensional biframes. In particular, we show that there is a bijection at the level of objects between proper Raney extensions and \emph{essential} strictly zero-dimensional biframes. Our approach is by no means the simplest possible way of proving this correspondence. Instead, we want to obtain the correspondence as a byproduct of a more general study of subcolocales of $\sll$ and subcolocales of $\So(L)$. We point out the following result.

\begin{proposition}\label{p: raney and szdbf are subcolocales}
Let $L$ be a frame. Raney extensions on $L$ are in bijective correspondence with subcolocales of $\So(L)$ containing all open sublocales. Strictly zero-dimensional biframes whose first component is $L$ are in bijective correspondence with dense subcolocales of $\sll$.    
\end{proposition}

Notice how, in the usual setting, all sublocales of $L$ are sober: for a sober space $X$, the sober subspaces are exactly the ones of the form
\[
\bca_i U_i\cup V_i^c
\]
where $U_i,V_i\se X$ are opens. Compare this with the fact that every sublocale $S\se L$ of a frame $L$ satisfies
\[
S=\bca\{\op(a)\ve \cl(b)\mid S\se \op(a)\ve \cl(b)\}
\]
to see that it is natural to view $\sll$ as the collection of all \emph{sober} subspaces of $L$. But Raney extensions and strictly zero-dimensional biframes tell us that if we want to capture all subspaces of $L$, not just sober ones, it suffices to look at subcolocales of $\So(L)$ or of $\sll$, rather than sublocales of $L$. For example, notice that if $X$ is a $T_0$ space it may be the case that distinct subspaces induce the same sublocale of $\Om(X)$. However, as Raney extensions faithfully extend $\bd{Top_0}$, this means that the two subspace inclusions will induce different surjections from $(\Om(X),\ca{U}(X))$ in $\bd{Raney}$. By the equivalence in Proposition \ref{p: raney and szdbf are subcolocales} above, this means that these will induce different subcolocales of $\So(\Om(X))$.

In conclusion, it may be argued that $\sll$ which is not refined enough to capture spaces which are not sober, but that the limitation is overcome by looking at the collections $\mf{S}(\So(L))$ and $\mf{S}(\sll)$, instead. In this paper, we show explicitly how to relate the two approaches. We do so by proving an adjunction between codense subcolocales of $\sll$ and certain subcolocales of $\So(L)$ which we call \emph{proper}. We prove that the adjunction maximally restricts to proper subcolocales of $\So(L)$ and a class of subcolocales $\kD\se \sll$ which we characterize explicitly.

\section{Preliminaries}

\subsection{The categories \texorpdfstring{$\bd{Frm}$}{frm} and \texorpdfstring{$\bd{Loc}$}{loc}} 

We first recall some background on frames and point-free topology. For more information on the categories of frames and locales, we refer the reader to Johnstone \cite{johnstone82} or the more recent \cite{PP12} and \cite{picado21}. A \emph{frame} is a complete lattice $L$ satisfying
\[
a\wedge \bigvee B = \bigvee \{a\wedge b\mid b\in B\},
\]
for all $a\in L$ and $B\subseteq L$. A \emph{frame homomorphism} is a function preserving arbitrary joins, including the bottom element $0$, and finite meets, including the top element $1$. We call $\bd{Frm}$ the category of frames and frame homomorphisms. Frames are complete Heyting algebras, with the Heyting implication computed as
\[
x\to y=\bve\{z\in L\mid z\we x\leq y\}.
\]
In particular, the \emph{pseudocomplement} of an $a\in L$ is the element $\neg a=a\to 0$. The archetypal example of frame is the lattice of open sets $\Omega(X)$ for a topological space $X$. The assignment $X\mapsto \Om(X)$ is the object part of a functor $\Omega:\bd{Top}\to \bd{Frm}^{op}$. An element $p$ of a frame $L$ is said to be \emph{prime} if whenever $x\we y\leq p$ for some $x,y\in L$, then $x\leq p$ or $y\leq p$. The collection of all primes of $L$ will be denoted by $\pt(L)$, and the assignment $L\mapsto \pt(L)$ is the object part of a functor $\pt:\bd{Top}\to \bd{Frm}^{op}$, which together with $\Om$ yields an adjunction $\Om:\bd{Top}\lra \bd{Frm}^{op}:\pt$ with $\Om\dashv \pt$. A frame is said to be \emph{spatial} if $a\nleq b$ for $a,b\in L$ implies that there is a prime $p$ with $b\leq p$ and $a\nleq p$. Spatial frames $L$ are precisely the fixpoints of the adjunction $\Om\dashv \pt$. This adjunction is idempotent, and so a frame is spatial if $L\cong \Omega(X)$ for some space $X$. Because of this adjunction, frames are regarded as pointfree spaces, but since the adjunction is contravariant sometimes the category $\bd{Loc}$, equivalent to $\bd{Frm}^{op}$, is used.

\subsubsection{The \texorpdfstring{$T_D$}{TD} approach}

An alternative to the classical duality is the so-called \emph{$T_D$-duality} from \cite{banaschewskitd}. For a space $X$ we call a point $x\in X$ a \emph{$T_D$-point} if it is the intersection of an open and a closed set. We define a space to be $T_D$ if all its points are $T_D$. The axiom is between $T_0$ and $T_1$ and it is introduced in \cite{Aull62}. For a frame $L$, a prime $p\in L$ is \emph{covered} if $\bwe_i x_i=p$ implies $x_i=p$ for some $i\in I$. We call $\pt_D(L)$ the set of covered primes of $L$. We say that a frame morphisms $f:L\to M$ is a \emph{D-morphism} if its right adjoint $f_*:M\to L$ maps covered primes to covered primes. We call $\bd{Frm_D}$ the category of frames and D-morphisms. The assignment $L\mapsto \pt_D(L)$ extends to a functor $\pt_D:\bd{Frm_D}^{op}\to \bd{Top}$. In \cite{banaschewskitd} it is shown that there is an adjunction $\Om:\bd{Top_D}\lra \bd{Frm_D}^{op}:\pt_D$ with $\Om\dashv \pt_D$, where $\bd{Top_D}$ is the category of $T_D$-spaces.
We call a frame \emph{$T_D$-spatial} if it is meet-generated by its covered primes. $T_D$-spatial frames are exactly the fixpoints of the adjunction above, namely the frames of opens of $T_D$-spaces.

\subsubsection{Sublocales}

For a frame $L$, a \emph{sublocale} of $L$ is a subset inclusion $S\subseteq L$ such that 
\begin{enumerate}
\item \label{subp1} $S$ is closed under arbitrary meets; 
\item \label{subp2} $a\to s\in S$ for all $a\in L$ and $s\in S$.  \end{enumerate}
Sublocales are frames when equipped with the order inherited from $L$. In fact, the name comes from the fact that such subset inclusions are, up to isomorphism, the regular monomorphisms in $\bd{Loc}$. If $S\se L$ is a sublocale, we call $\bwe^S$ and $\we^S$ the arbitrary and the binary meets in $S$, respectively, and we use a similar convention for joins. Whenever $L$ is any lattice and $M\se L$ a lattice with the inherited order, we use analogous notation for lattice operations in $M$ with the inherited order. Sublocales have closure operators associated to them. For a sublocale $S\se L$ the map $\nu_S(a)=\bwe\{s\in S\mid a\leq s\}$ for all $a\in L$ is called its \emph{nucleus}.

\begin{lemma}
    Let $S\se L$ be a sublocale. Then, for $s_i\in S$;
    \begin{enumerate}
        \item $\bwe^S_i s_i=\bwe_i s_i$;
        \item $\bve^S_i s_i=\nu_S(\bve_i s_i)$.
    \end{enumerate}
\end{lemma}

 Every sublocale also has an associated frame \emph{congruence}, a binary relation on $L$ which is a subframe of $L\times L$. An alternative approach, which we will adopt, is to consider the bijection between sublocales of $L$ and \emph{precongruences} on $L$. These are defined to be binary relations $R$ on $L$ such that
\begin{enumerate}
    \item $R$ is reflexive;
    \item $R$ is transitive;
    \item $a'\leq a$ and $(a,b)\in R$ and $b\leq b'$ implies $(a',b')\in R$;
    \item $(a_i,b)\in R$ implies $(\bve_i a_i,b)\in R$;
    \item $(a,b_1),(a,b_2)\in R$ implies $(a,b_1\we b_2)\in R$.
\end{enumerate}
This correspondence is introduced in \cite{moshier18}. The sublocale $S$ is associated with the precongruence $\{(a,b)\in L\times L\mid \nu_S(a)\leq \nu_S(b)\}$. Conversely, for a precongruence $R\se L\times L$, the associated sublocale is
\[
\bigcap \{\cl(x)\ve \op(y)\mid (x,y)\in R\}.
\]
A sublocale of $L$ is \emph{dense} if it contains $0$. Dense sublocales, then, are closed under arbitrary intersections. For a frame $L$, the smallest dense sublocale $\bl(0)\se L$ is always Boolean, and it is called its \emph{Booleanization}.

\subsection{Subcolocales and the coframe \texorpdfstring{$\sll$}{sl}}

\subsubsection{Subcolocales}

Coframes come equipped with a \emph{co-Heyting operator}, known as the \emph{difference} $x{\sm}y$ of two elements $x,y\in C$, computed as
\[
x{\sm} y=\bwe\{z\in C\mid x\leq y\vee z\}.
\]
This operator is characterized by the condition that $x{\sm}y\leq z$ if and only if $x\leq y\vee z$. In particular, the \emph{supplement} of $c\in C$ is $c^{*}= 1{\sm}c$. We shall freely use some its properties, listed in the following lemma.

\begin{lemma}\label{l: properties of difference}
Let $C$ be a coframe. For elements $c,d,c_i,x\in C$:
    \begin{enumerate}
     \item \label{properties of difference 1}If $x$ is complemented, $c{\sm} x=c\we x^{*}$;
     \item \label{properties of difference 2}$(\bve_i c_i){\sm}d=\bve_i (c_i{\sm}d)$;
     \item \label{properties of difference 3}$d{\sm} \bwe_i c_i=\bigvee_i (d{\sm} c_i)$.
    \end{enumerate}
\end{lemma}
For a coframe $C$, we say that an element $c\in C$ is \emph{linear} if $\bve_i (a_i\we c)=\bve_i a_i \we c$ for all $a_i\in C$. 
\begin{lemma}\label{l: linear}
    Complemented elements of a coframe are linear.
\end{lemma}
Of particular importance will be the notion dual to that of sublocale. Let us define it explicitly. For a coframe $C$, a \emph{subcolocale} is an inclusion $D\se C$ such that
\begin{enumerate}
    \item $D$ is closed under all joins;
    \item $d{\sm}c\in D$ for all $d\in D$ and $c\in C$.
\end{enumerate}
An inclusion $D\se C$ for a frame $D$ is a subcolocale iff $D^{op}\se C^{op}$ is a sublocale. Dualizing the analogous notion for frames, we see that a subcolocale $D\se C$ determines an interior operator $\nu_D:C\to C$, which we call its \emph{conucleus}.

\begin{lemma}\label{l: subcolocales very basic}
    Let $D\se C$ be a sublocale. Then, for $d_i\in D$;
    \begin{enumerate}
        \item $\bwe^D_i d_i=\nu_D(\bwe_i d_i)$;
        \item $\bve^D_i d_i=\bve_i d_i$.
    \end{enumerate}
\end{lemma}

We say that a subcolocale is \emph{codense} if it contains $1$. We keep the term for the dual notion and call the smallest codense subcolocale of a coframe its \emph{Booleanization}.

\subsubsection{The coframe of sublocales}

The family $\mathsf{S}(L)$ of all sublocales of $L$, ordered by inclusion, is a coframe. Meets are set-theoretical intersections. Because $\sll$ is a coframe, it also comes with a difference operation, computed as $S{\sm}T=\bca\{U\in \sll\mid S\se T\vee U\}$. This is studied in \cite{remainders}. For each $a\in L$, there is an \emph{open sublocale} $\mathfrak{o}(a)=\{b\in L \mid b=a\to b\}=\{a\to b\mid b\in L\} $ and a \emph{closed sublocale} $\mathfrak{c}(a)={\uparrow}a$. Open and closed sublocales behave like open and closed subspaces in many respects, in the lemma below we list a few.

\begin{lemma}
For every frame $L$ and $a,b,a_i\in L$ we have
\begin{enumerate}
    \item \label{opcl1}$\op(1)=L$ and $\op(0)=\{1\}$.
    \item \label{opcl2}$\cl(1)=\{1\}$ and $\cl(0)=L$.
    \item \label{opcl3} $\op(a)\cap \cl(a)=\{1\}$ and $\op(a)\ve \cl(a)=L$.
    \item \label{opcl4}$\bve_i \op(a_i)=\op(\bve_i a_i)$ and $\op(a)\cap \op(b)=\op(a\we b)$.
    \item \label{opcl5}$\bca_i \cl(a_i)=\cl(\bwe_i a_i)$ and $\cl(a)\ve \cl(b)=\cl(a\we b)$.
\end{enumerate}
\end{lemma}

In particular, by \ref{opcl3}, open and closed sublocales are complemented, and as such (Lemma \ref{l: linear}) they are linear. Intersections of open sublocales are called \emph{fitted} sublocales. These form a subcoframe of $\sll$, which we denote as $\So(L)$. The closure operator associated with it is called the \emph{fitting}. This is studied in \cite{clementino18}. More explicitly, the corestriction to its fixpoints is:
\begin{align*}
 \mi{fit}:&\sll\to \sll   \\
          &S\mapsto \bca\{\op(x)\mid x\in L,S\se \op(x)\}.
\end{align*}
Another important class of sublocales is that of the \emph{two-element sublocales}. These are sublocales of the form $\{1,p\}$ for some element $p\in L$, which is then necessarily prime. 

\subsubsection{Exactness and strong exactness} 

For a complete lattice $L$ we say that a join (resp. a meet) is \emph{exact} if it distributes over binary meets (resp. joins). These notions are compared in \cite{ball14}. When $L$ is a frame, we also have the notion of \emph{strongly exact} meet: a meet $\bwe_i x_i$ such that $x_i\to y=y$ for all $i\in I$ implies that $\bwe_i x_i\to y=y$. A filter $F\se L$ of a frame is called \emph{exact} if it is closed under exact meets, and \emph{strongly exact} if it is closed under strongly exact meets. We call $\fse(L)$ the ordered collection of strongly exact filters of $L$, and $\fe(L)$ the collection of the exact filters. 
The following is Theorem 3.5 in \cite{moshier20}. 
\begin{lemma}\label{l: SoL and FSEL are isomorphic}
There is an isomorphism $\varphi:\So(L)\cong \fse(L)$ given by
    \[
     \varphi(F)=\{x\in L\mid F\se \op(x)\}
    \]
    for each $F\in \So(L)$.
\end{lemma}
We will consider the map $\varphi\circ \mi{fit}:\sll\to \fse(L)$. We call this $\mi{ker}$, for \emph{kernel}, as it assigns to a sublocale $S\se L$ the set $s^{-1}(1)$, where $s:L\to S$ is its surjection. Theorem 6.6 of \cite{jakl25}, together with the fact that $\varphi$ is an isomorphism, gives the following.

\begin{lemma}\label{l: FEL are the kernels of SbL}
For every frame $L$, $\mi{ker}[\Sb(L)]=\fe(L)$.
\end{lemma}

We say that a frame map $f:L\to M$ is \emph{exact} if, whenever $\bwe_ix_i\in L$ is an exact meet, $\bwe_i f(x_i)$ is exact, and $\bwe_i f(x_i)=f(\bwe_i x_i)$. A sublocale $S\se L$ whose surjection $s:L\to S$ is exact is called an \emph{exact} sublocale. Exact sublocales are precisely those whose surjection preserves all exact meets. The next result is Proposition 7.15 in \cite{suarez25}.

\begin{lemma}\label{l: characterization of exact sublocales}
    A sublocale $S\se L$ is exact if and only if whenever $\bwe_i x_i\in L$ is an exact meet and $S\cap \cl(x_i)\se \cl(x)$ for every $i\in I$ then $S\cap \cl(\bwe_i x_i)\se \cl(x)$ for all $x\in L$.
\end{lemma}

\subsubsection{Distinguished subcolocales of \texorpdfstring{$\sll$}{SL}}

In this work, we look at subcolocales of $\sll$. Among these is the collection $\mf{S}_b(L)$ of joins of complemented sublocales (see \cite{arrieta22}). These also coincide with those sublocales of the form $\bve_i \cl(x_i)\cap \op(y_i)$. We observe that a subcolocale of $\sll$ is codense if and only if it contains $L$. The following is well-known.

\begin{lemma}\label{l: SbL is codense}
For any frame $L$, the inclusion $\mf{S}_b(L)\se \sll$ is a codense subcolocale. In particular, it is the Booleanization of $\sll$.    
\end{lemma}

We will need the following result.

\begin{lemma}\label{l: smooth implies induced}
If $X$ is a $T_D$-space, all sublocales in $\Sb(\Om(X))$ are $T_D$-spatial.  
\end{lemma}
\begin{proof}
  Proposition 3.2 of \cite{arrieta22} states that all sublocales in $\Sb(\Om(X))$ are induced by some subspace of $X$, in the sense that, for each of these, the surjection corresponding to them is of the form $\Om(i_Y):\Om(X)\to \Om(Y)$ for some subspace inclusion $i_Y:Y\se X$. Since subspaces of $T_D$ spaces are $T_D$, these are $T_D$-spatial sublocales. 
\end{proof}

We will also look at the collection $\mf{S}_{sp}(L)$ of spatial sublocales of $L$, by which we mean the sublocales $S\se L$ where $S$ is a spatial frame. Equivalently, spatial sublocales are characterized as those which are joins of two-element sublocales. For all frames $L$, the inclusion $\mf{S}_{sp}(L)\se \sll$ is a subcolocale inclusion, in particular, it is the subcolocale associated with the spatialization surjection of $\sll^{op}$ (Proposition 3.14 of \cite{suarez22}). The next result follows by definition of codensity.

\begin{lemma}\label{l: SspL is codense iff L is spatial}
   For a frame $L$, the subcolocale $\mf{S}_{sp}(L)\se \sll$ is codense if and only if $L$ is spatial.
\end{lemma}
The collection of all exact sublocales of $L$ will be denoted as $\See(L)$. The coming result follows from Theorem 7.20 of \cite{suarez25} and Proposition 7.12 from the same paper. The second part of the claim follows from $\Sb(L)\se \sll$ being the smallest codense subcolocale.
\begin{lemma}\label{l: SbL included in SeeL}
    The inclusion $\See(L)\se \sll$ is a codense subcolocale. In particular, $\Sb(L)\se \See(L)$.
\end{lemma}

\subsection{Strictly zero-dimensional biframes and Raney extensions}

\subsubsection{Strictly zero-dimensional biframes}

 A \emph{biframe} (see \cite{banaschewski83}) is a triple $\ca{L}=(L_1,L_2,L)$ where $L$ is a frame and $L_1,L_2\se L$ are subframes such that $L_1\cup L_2$ generates $L$ in the sense that $L$ is the smallest subframe containing $L_1\cup L_2$. An element $a_1\in L_1$ is said to be \emph{bicomplemented} if it is complemented in $L$ and its complement is in $L_2$, and similarly for elements of $L_2$. A biframe is said to be \emph{zero-dimensional} if both $L_1$ and $L_2$ are join-generated by their bicomplemented elements. A biframe is said to be \emph{strictly zero-dimensional} if it is zero-dimensional and all elements of $L_1$ are bicomplemented. Strictly zero-dimensional biframes are studied in detail in \cite{manuell15}. We will later see (Lemma \ref{l: very easy}) that dense subcolocales of $\sll$ contain all closed sublocales. In particular, these are embedded as a subcoframe $\cl[L]\se \kD$. We give an equivalent description to the category of strictly zero-dimensional biframes, similar to the one given in Section 3.2 of \cite{manuell15}. The category $\bd{SZDBF}$ is the category:
\begin{enumerate}
    \item Whose objects are pairs $(L,\mk{D})$ where $L$ is a frame and $\mk{D}\se \sll$ a dense subcolocale;
    \item Whose morphisms $f:(L,\kD)\to (M,\mk{E})$ are frame maps $f:L\to M$ such that the anti-isomorphic coframe map $\cl(f):\cl[L]\to \cl[M]$ extends to a coframe map $\overline{f}:\mk{D}\to \mk{E}$.
\end{enumerate}

\subsubsection{Raney extensions}

Raney extensions are introduced in \cite{suarez25}. These structures are inspired by a duality studied in \cite{bezhanishvili20}, based on work by Raney, see \cite{raney52}. Raney extensions may be identified with pairs $(L,\ca{F})$, where $\ca{F}\se \fse(L)^{op}$ is a subcolocale containing all principal filters (see Theorem 3.9 in \cite{suarez25}). Here, we use the isomorphism in \ref{l: SoL and FSEL are isomorphic} to give an equivalent definition. We define $\bd{Raney}$ to be the category:
\begin{enumerate}
    \item Whose objects are pairs $(L,\kF)$ where $L$ is a frame and $\kF\se \So(L)$ a subcolocale containing all open sublocales;
    \item Whose morphisms $f:(L,\kF)\to (M,\mk{G})$ are frame maps $f:L\to M$ such that the isomorphic frame map $\op(f):\op[L]\to \op[M]$ extends to a coframe map $\overline{f}:\kF\to \mk{G}$.
\end{enumerate}

\section{The main adjunction}
\subsection{From subcolocales of \texorpdfstring{$\sll$}{SL} to subcolocales of \texorpdfstring{$\So(L)$}{SoL} via fitting}

In this subsection, we show that the fitting operator $\mi{fit}:\sll\to \So(L)$ is such that its direct image $\mi{fit}[-]:\ca{P}(\sll)\to \ca{P}(\So(L))$ sends codense subcolocales to special kinds of subcolocales of $\So(L)$. First, we work towards a characterization of codense subcolocales of $\sll$. 

\begin{proposition}\label{p: subcolocales of SL characterization}
    An inclusion $\kD\se \sll$ is a subcolocale if and only if it is closed under arbitrary joins and stable under both $-\cap \op(x)$ and $-\cap \cl(x)$ for all $x\in L$.
\end{proposition}
\begin{proof}
  Suppose that $\kD$ satisfies the assumptions in the statement. By definition of subcolocale, it suffices to show that $S\in \kD$ implies $S{\sm}T\in \kD$ for all $T\in \sll$. Let $S\in \kD$ and let $T=\bca_i \op(x_i)\vee \cl(y_i)$. By assumption, for every $i\in I$ we have $S\cap \cl(x_i)\cap \op(y_i)\in \kD$. By \ref{properties of difference 1} of Lemma \ref{l: properties of difference}, this equals $S{\sm}(\op(x_i)\ve \cl(y_i))$. As $\kD$ is closed under all joins and by item \ref{properties of difference 3} of Lemma \ref{l: properties of difference}, $S{\sm}\bca_i \op(x_i)\ve \cl(y_i)\in \kD$, as desired. For the other direction, if $\kF\se \sll$ is a subcolocale, indeed it must be stable under $-\cap \op(x)$ and $-\cap \cl(x)$, as by item \ref{properties of difference 1} of Lemma \ref{l: properties of difference} these are the same as $-{\sm} \cl(x)$ and $-{\sm}\op(x)$, respectively.
\end{proof}

\begin{lemma}\label{l: very easy}
   Let $\kD\se \sll$ be a codense subcolocale. Then:
   \begin{enumerate}
       \item \label{very easy 1}$\kD$ contains all open and closed sublocales;
       \item \label{very easy 2}$\kD$ is meet-generated by the elements of the form $\op(x)\vee \op(y)$.
   \end{enumerate}
\end{lemma}
\begin{proof}
  If $\kD\se \sll$ is codense, by definition it must contain $L$. By Lemma \ref{p: subcolocales of SL characterization}, then, it must contain $L{\sm}\op(x)$ and $L{\sm}\cl(x)$ for all $x\in L$, but by \ref{properties of difference 1} of Lemma \ref{l: properties of difference} these equal $\cl(x)$ and $\op(x)$, respectively. Every element of $\kD$ is $\nu_{\kD}(S)$ for some sublocale $S\in \sll$, and $S=\bca_i \op(x_i)\ve \cl(y_i)$ for some $x_i,y_i\in L$. Then, every element of $\kD$ can be written as $\nu_{\kD}(\bca_i \op(x_i)\ve \cl(y_i))$ for some $x_i,y_i\in L$. This equals $\bwe^{\kD}_i\nu_{\kD}(\op(x_i)\ve \cl(x_i))$ by Lemma \ref{l: subcolocales very basic}. As $\kD$ contains all open and closed sublocales by the first item, and as it is closed under joins by \ref{p: subcolocales of SL characterization}, $\nu_{\kD}(\op(x_i)\ve \cl(y_i))=\op(x_i)\ve \cl(y_i)$, and our claim is proven.
\end{proof}

We now characterize the subcolocales of $\So(L)$ in a similar manner. We call $\bve^{\mi{fit}}$ and $\sm^{\mi{fit}}$ the joins and the coframe differences in $\So(L)$, respectively.
\begin{lemma}\label{l: difference in SoL}
    In $\So(L)$, for all $F_i,F,G$:
    \begin{enumerate}
        \item $\bve^{\mi{fit}}_i F_i=\mi{fit}(\bve_i F_i)$;
        \item  $F{\sm}^{fit}G=\mi{fit}(F{\sm}G)$.
    \end{enumerate}
\end{lemma}
\begin{proof}
For the joins, it suffices to notice that if $F_i\se F$ is an upper bound in $\So(L)$, as $F$ is a fixpoint of $\mi{fit}$ also $\mi{fit}(\bve_i F_i)\se F$. By its definition, the difference $F{\sm}^{fit}G$ is such that $F{\sm}^{fit}G\se \op(x)$ if and only if $F\se G\vee \op(x)$ for all $x\in L$. But the following are also equivalent:
\begin{prooftree}
\AxiomC{$F\se G\vee \op(x)$}
\UnaryInfC{$F{\sm}G\se \op(x)$}
\UnaryInfC{$\mi{fit}(F{\sm}G)\se \op(x))$.}
\end{prooftree}
Since fitted sublocales are intersections of open ones, and we have shown that the fitted sublocales $F{\sm}^{fit}G$ and $\mi{fit}(F{\sm}G)$ are contained in the same opens, they must be equal.
\end{proof}

\begin{lemma}\label{l: fit of F cap cx}
    For a frame $L$ and $F\in \So(L)$, and $x\in L$
    \[
    \mi{fit}(F\cap \cl(x))=F{\sm}^{fit}\op(x).
    \]
\end{lemma}
\begin{proof}
    By \ref{properties of difference 1} of Lemma \ref{l: properties of difference}, $F\cap \cl(x)=F{\sm}\op(x)$, and so $\mi{fit}(F\cap \cl(x))=\mi{fit}(F{\sm}\op(x))$. By Lemma \ref{l: difference in SoL}, $\mi{fit}(F{\sm}\op(x))=F{\sm}^{\mi{fit}}\op(x)$.
\end{proof}

\begin{proposition}\label{p: subcolocales of SoL characterization}
    A collection $\mk{F}\se \So(L)$ is a subcolocale if and only if it is closed under arbitrary joins and is stable under $\mi{fit}(-\cap \cl(x))$ for all $x\in L$.
\end{proposition}
\begin{proof}
    Suppose that $\kF\se \So(L)$ satisfies the assumptions in the claim. We check that $F\in \kF$ implies $F{\sm}^{fit}\bca_i \op(x_i)\in \kF$ for all families $x_i\in L$. If $F\in \kF$, by our assumption, $\mi{fit}(F\cap \cl(x_i))\in \kF$ for every $i\in I$. By Lemma \ref{l: fit of F cap cx}, this means that $F{\sm}^{fit}\op(x_i)\in \kF$. By \ref{properties of difference 3} of Lemma \ref{l: properties of difference}, $F{\sm}^{fit}\bca_i \op(x_i)=\bve_i F{\sm}^{fit}\op(x_i)$, and this is in $\kF$ as we assumed this collection is closed under all joins. Conversely, if $\kF$ is a subcolocale, it is closed under all joins by definition, and by definition also $F{\sm}^{\mi{fit}}\op(x)\in \kF$ for all $x\in L$. By Lemma \ref{l: fit of F cap cx}, then, $\mi{fit}(F\cap \cl(x))\in \kF$ for all $x\in L$. 
\end{proof}

Finally, we want to show that for every codense subcolocale $\kD\se \sll$ the inclusion $\mi{fit}[\kD]\se \So(L)$ is a subcolocale containing all opens. We also look at how the conuclei of these interact.

\begin{lemma}\label{o:wF}
Let $\kD\se \sll$ be a codense subcolocale. For every $S\in \sll$ and $x\in L$, 
\begin{enumerate}
    \item $\nu_{\kD}(S\cap \op(x))=\nu_{\kD}(S)\cap \op(x)$;
    \item $\nu_{\kD}(S\cap \cl(x))=\nu_{\kD}(S)\cap \cl(x)$.
\end{enumerate}

\end{lemma}
\begin{proof}
    The claims holds because $\op(x),\cl(x)\in \kD$ by item \ref{very easy 1} of Lemma \ref{l: very easy}, and $\kD$ is stable under $-\cap \op(x)$ and $-\cap \cl(x)$ by Proposition \ref{p: subcolocales of SL characterization}.
\end{proof}

\begin{lemma}\label{l: intersection and fitting}
    For a sublocale $S\se L$, $\mi{fit}(S\cap \cl(x))=\mi{fit}(\mi{fit}(S)\cap \cl(x))$ for all $x\in L$. 
\end{lemma}
\begin{proof}
The following are equivalent statements for all $y\in L$. At each step we only use basic properties of fitting and of open and closed sublocales.
    \begin{prooftree}
\AxiomC{$\mi{fit}(S\cap \cl(x))\se \op(y)$}
\UnaryInfC{$S\cap \cl(x)\se \op(y)$}
\UnaryInfC{$S\se \op(y)\ve \op(x)$}
\UnaryInfC{$\mi{fit}(S)\se \op(y)\ve \op(x)$}
\UnaryInfC{$\mi{fit}(S)\cap \cl(x)\se \op(y)$.}
\end{prooftree}
Indeed, then, $\mi{fit}(S\cap \cl(x))=\mi{fit}(\mi{fit}(S)\cap \cl(x))$ as desired.
\end{proof}

\begin{proposition}\label{p: from denseSL to denseSoL}
    Let $\kD\se \sll$ be a codense subcolocale.
    \begin{enumerate}
        \item \label{from dense SL item 1}The collection $\kF:=\mi{fit}[\kD]\se \So(L)$ is a subcolocale containing all opens;
        \item \label{from dense SL item 2}$\nu_{\kF}(F)=\mi{fit}(\nu_{\kD}(F))$ for all $F\in \So(L)$;
        \item \label{from dense SL item 3}$\bwe^{\kF}F_i=\mi{fit}(\nu_{\kD}(\bca_i F_i)$ for $D_i\in \kD$.
    \end{enumerate}
    
\end{proposition}
\begin{proof}
Let us prove the three items in turn.
\begin{enumerate}
    \item We use the characterization in Proposition \ref{p: subcolocales of SoL characterization}. Closure of $\mi{fit}[\kD]\se \So(L)$ under all joins follows from the fact that fitting is a closure operator on $\sll$, and so the map $\mi{fit}:\sll\to \So(L)$ to its fixpoints preserves all joins. Next, we have to show that for $F\in \mi{fit}[\kD]$ the element $\mi{fit}(F\cap \cl(x))$ is in $\mi{fit}[\kD]$, by Proposition \ref{p: subcolocales of SoL characterization}. This holds because for $D\in \kD$ such that $F=\mi{fit}(D)$ we have $D\cap \cl(x)\in \kD$, by Proposition \ref{p: subcolocales of SL characterization} and $\mi{fit}(D\cap \cl(x))=\mi{fit}(\mi{fit}(D)\cap \cl(x))$ by Lemma \ref{l: intersection and fitting}. Then, $\mi{fit}[\mk{D}]\se \So(L)$ is a subcolocale. Since $\op(x)\in \mk{D}$ for all $x\in L$, by \ref{very easy 1} of Lemma \ref{l: very easy}, $\mi{fit}[\mk{D}]$ contains all opens.  
    \item The following are equivalent statements. Again, for each derivation we are only using basic facts about fitting and the conucleus $\nu_{\kD}$.
      \begin{prooftree}
      \AxiomC{$\mi{fit}(\nu_{\kD}(F))\se \op(x)$}
      \UnaryInfC{$\nu_{\kD}(F)\se \op(x)$}
      \UnaryInfC{$\text{$\mi{fit}(D)\se F$ implies $\mi{fit}(D)\se \op(x)$ for all $D\in \kD$}$}
      \UnaryInfC{$\text{$D\se F$ implies $D\se \op(x)$ for all $D\in \kD$}$}
      \UnaryInfC{$\nu_{\kD}(\mi{fit}(F))\se \nu_{\kD}(\op(x))$}
      \UnaryInfC{$\nu_{\kD}(\mi{fit}(F))\se \op(x)$.}
  \end{prooftree}
    \item This follows from (2) and from Lemma \ref{l: subcolocales very basic}.\qedhere
\end{enumerate}

\end{proof}
We have found a (clearly monotone) map $\mi{fit}[-]:\ca{CD}(\sll)\to \ca{C}(\So(L))$ from codense subcolocales of $\sll$ to subcolocales of $\So(L)$ containing all opens. We want to show a concrete example of the assignment $\mi{fit}[-]:\ca{CD}(\sll)\to \ca{C}(\So(L))$ not being injective in general.

\begin{lemma}\label{l: SbL and SeeL have the same fitting}
    For a frame $L$, $\mi{fit}[\Sb(L)]=\mi{fit}[\See(L)]$.
\end{lemma}
\begin{proof}
   Since $\Sb(L)\se \See(L)$ by Lemma \ref{l: SbL included in SeeL}, $\mi{fit}[\Sb(L)]\se \mi{fit}[\See(L)]$. Let us show the reverse inclusion. By Lemma \ref{l: SoL and FSEL are isomorphic}, it suffices to show that $\mi{ker}[\See(L)]\se \mi{ker}[\Sb(L)]$, and since $\mi{ker}[\Sb(L)]=\fe(L)$, by Lemma \ref{l: FEL are the kernels of SbL}, it suffices to show that if $E\in \See(L)$ then $\mi{ker}(E)$ is an exact filter. Let $\bwe x_i\in L$ be an exact meet. If $E\se \op(x_i)$, then we have $E\cap \cl(x_i)\se \{1\}$. By the characterization of exact sublocales in Lemma \ref{l: characterization of exact sublocales}, this implies that $E\cap \cl(\bwe_i x_i)\se \{1\}$, that is, $E\se \op(\bwe_i x_i)$, and so, indeed, $\bwe_i x_i\in \mi{ker}(E)$.
\end{proof}

By Lemma \ref{l: SbL and SeeL have the same fitting}, above, then, to show that $\mi{fit}[-]$ is not injective, it suffices to find a frame $L$ where $\See(L)\nsubseteq \Sb(L)$. Complete sublocales, i.e. sublocales such that their surjection preserves all meets, in particular are exact (Proposition 7.14 in \cite{suarez25}). Then, complete sublocales which are not in $\Sb(L)$ are witnesses of $\See(L)\nsubseteq \Sb(L)$. Example 5.12 in \cite{bezhanishvili2025mckinseytarskialgebrasraneyextensions} presents a concrete example of this. The following class of examples of complete sublocales which are not in $\Sb(L)$ is due to Igor Arrieta, who we thank.

 \begin{example}\label{e: igor example}
For every frame $L$, there is a frame surjection $\varepsilon:\ca{D}(L)\to L$ defined as $\varepsilon(D)=\bve D$. We claim that if $L$ is completely distributive and is not $T_D$-spatial, then the sublocale $\varepsilon_*[L]\se \ca{D}(L)$ associated with this surjection is exact but not in $\Sb(L)$. An example of such a completely distributive lattice is given, for example, by the interval $[0,1]\se \mathbb{R}$, which has no covered primes. By complete distributivity of $L$, the surjection $\varepsilon:\ca{D}(L)\to L$ preserves all meets, and so it corresponds to an exact sublocale. Let us show that this sublocale is not in $\Sb(L)$. We note that $\ca{D}(L)$ is $T_D$-spatial: the covered primes are exactly the elements of the form $L{\sm}\up x$ for some $x\in L$, and these meet-generate $\ca{D}(L)$. Then, by Lemma \ref{l: smooth implies induced}, every sublocale in $\Sb(L)$ is $T_D$-spatial. As $L$ is not $T_D$-spatial, by assumption, the sublocale corresponding $\varepsilon_*[L]\se \ca{D}(L)$ cannot be in $\Sb(L)$. 
\end{example}

\subsection{From proper subcolocales of \texorpdfstring{$\So(L)$}{SoL} to subcolocales of \texorpdfstring{$\sll$}{SL}}

In this section, we restrict to a particular class of subcolocales of $\So(L)$ and define for their collection a left order adjoint to (the suitable corestriction of) the monotone map $\mi{fit}[-]:\ca{CD}(\sll)\to \ca{C}(\So(L))$. We say that a subcolocale $\kF\se \So(L)$ is \emph{proper} if it contains all open sublocales and the join $\bve_i^{\kF}\op(x_i)$ is exact for each family $x_i\in L$.

\begin{lemma}\label{l: from denseSL to properdenseSoL}
    If $\kD\se \sll$ is a codense subcolocale, then $\mi{fit}[\kD]\se \So(L)$ is proper.
\end{lemma}
\begin{proof}
   By Proposition \ref{p: from denseSL to denseSoL}, $\mi{fit}[\kD]\se \So(L)$ is a subcolocale containing all opens. Let us show that for all $x_i\in L$ and $F\in \So(L)$:
    \[
    \op(\bve_i x_i)\wF F\se \bve_i^{\kF}\op(x_i)\wF F.
    \]
    Suppose that $\op(x_i)\wF F\se \op(y)$ for all $i\in I$. Then, by item \ref{from dense SL item 2} of Proposition \ref{p: from denseSL to denseSoL}, $\mi{fit}(\nu_{\kD}(\op(x_i)\cap F))\se \op(y)$, and by Lemma \ref{o:wF} this implies $\op(x_i)\cap \nu_{\kD}(F)\se \op(y)$. This also means $\nu_{\kD}(F)\se \op(y)\ve \cl(x_i)$ for all $i\in I$, that is, $\nu_{\kD}(F)\se \op(y)\ve \cl(\bve_i x_i)$. Thus, $\op(\bve_i x_i)\cap \nu_{\kD}(F)\se \op(y)$, and this implies $\mi{fit}(\nu_{\kD}(\op(\bve_i x_i)\cap F))\se \op(y)$, where we have used \ref{o:wF} again. As the left-hand side is $\op(\bve_i x_i)\wF F$, by item \ref{from dense SL item 3} of Proposition \ref{p: from denseSL to denseSoL}, our claim is proven.
\end{proof}
\begin{corollary}\label{c: concrete examples of proper}
    The following are all proper subcolocales.
    \begin{enumerate}
         \item $\So(L)\se \So(L)$ for every frame $L$;
        \item $\mi{fit}[\Sb(L)]\se \So(L)$ for every frame $L$;
        \item $\mi{fit}[\mf{S}_{sp}(L)]\se \So(L)$ for every spatial frame $L$.
    \end{enumerate}
\end{corollary}
\begin{proof}
We prove each item using \ref{p: from denseSL to denseSoL}. 
For the first, we just note $\So(L)=\mi{fit}[\mf{S}(L)]$. The inclusion $\mf{S}_b(L)\se \sll$ is codense by Lemma \ref{l: SbL is codense}. Finally, for a spatial frame $L$, the inclusion $\mf{S}_{sp}(L)\se \sll$ is codense by Lemma \ref{l: SspL is codense iff L is spatial}.
\end{proof}

For every subcolocale $\ca{F}\se \So(L)$ containing all opens, for all $F\in \ca{F}$ we define the relation $\leq^{\ca{F}}_{F}\se L\times L$ as:
  \[
     \leq^{\ca{F}}_F=\{(x,y)\in L\times L\mid F\wF\op(x)\se \op(y)\}.
\]
Should $\ca{F}\se \So(L)$ be clear from the context, we will sometimes abbreviate this as $\leq_F$.

\begin{proposition}\label{p:frmcong}
     A subcolocale $\mk{F}\se \So(L)$ containing all opens is proper if and only if for every $F\in \kF$ the relation $\leq^{\ca{F}}_F$ is a frame precongruence.
\end{proposition}
\begin{proof}
    The only condition which is not shown with routine computations is stability under arbitrary joins. If $x_i\leq_F y$ for $F\in \kF$, then $\bvf_i (F\wF\op(x_i))\se\op(y_i)$. The desired result follows by exactness of the join $\bvf_i\op(x_i)=\op(\bve_i x_i)$. For the converse, suppose that there is a subcolocale $\kF\se \So(L)$ containing all opens, which is not proper. Let $F\in \kF$ and $x_i\in L$ with $ F\wF\op(\bve_i x_i)\nsubseteq\bvf_i (F\wF\op(x_i))$. So, there is $y\in L$ with $F\wF\op(x_i)\se \op(y)$ for all $i\in I$ but $F\wF\op(\bve_i x_i)\nsubseteq \op(y)$. The first set inclusion means $x_i\leq_F y$ for each $i\in I$, and the second means $\bve_i x_i\nleq_F y$. Then, the relation $\leq_F$ is not stable under joins, and so it is not a frame precongruence.
\end{proof}

To our ends, the following characterization of proper subcolocales will be more useful. 

\begin{corollary}\label{c: charsigma}
    A subcolocale $\mk{F}\se \So(L)$ containing all opens is proper if and only if for every $F\in \mk{F}$ there is $\sigma(F)\in \sll$ such that 
\begin{align}
\sigma(F)\se \cl(x)\vee \op(y)\text{ if and only if }x\leq_F y,    
\end{align}

which is necessarily unique. Equivalently, $\sigma(F)$ is such that
\begin{align}
    \mi{fit}(\sigma(F)\cap \op(x))=F\wF\op(x)\text{ for each $x\in L$. } 
\end{align}

\end{corollary}
\begin{proof}
The first claim follows from \ref{p:frmcong}, and by the correspondence between precongruences and sublocales. To see that the second condition is equivalent to the first, we note that the first condition amounts to having $\sigma(F)\cap \op(x)\se \op(y)$ if and only if $x\leq_F y$. In turn, this is equivalent to having $\sigma(F)\cap \op(x)\se \op(y)$ if and only if $F\wF \op(x)\se \op(y)$.
\end{proof}
For each proper subcolocale $\kF\se \So(L)$ we may then define a map 
\begin{align*}
   \sigma_{\kF}: &\kF\to \sll\\
    & F \mapsto \bca\{\op(x)\ve \cl(y)\mid x,y\in L, x\leq_F y\}.
\end{align*}

When $\kF$ is clear from the context, we simply call this map $\sigma$. 

Let us see a few basic facts about this map.

\begin{lemma}\label{l: sigma basic}
Let $\kF\se \So(L)$ be a proper subcolocale. Then, for each $F\in \kF$ and $x\in L$:
\begin{enumerate}
    \item \label{sigma basic 1}$\mi{fit}(\sigma(F))=F$;
     \item \label{sigma basic 2} $\sigma(\op(x))=\op(x)$;
     \item \label{sigma basic 3}$\sigma (F\wF \op(x))=\sigma(F)\cap \op(x)$.
\end{enumerate}
\end{lemma}
\begin{proof}
Suppose that $\kF\se \So(L)$ is a proper codense subcolocale.
\begin{enumerate}

    \item By Corollary \ref{c: charsigma}, $\mi{fit}(\sigma(F))=\mi{fit}(\sigma(F)\cap \op(1))=F\wF\op(1)=F$. 

     \item We use Corollary \ref{c: charsigma}. We have to check $\mi{fit}(\op(x)\cap \op(y))=\op(x)\wF \op(y)$ for all $y\in L$. Indeed, $\mi{fit}(\op(x)\cap \op(y))=\op(x)\cap \op(y)=\op(x\we y)$, and this equals $\op(x)\wF\op(y)$ as $\kF$ contains all open sublocales.

        \item We use Corollary \ref{c: charsigma} again. We have to show that 
    \setcounter{equation}{0}
    \begin{equation}\label{sigma equation 1}
         \mi{fit}(\sigma(F)\cap \op(x)\cap \op(y))=F\wF\op(x)\wF\op(y)
    \end{equation}
    for all $y\in L$. By the characterization of $\sigma(F)$ from \ref{c: charsigma}. 
    \begin{equation}\label{sigma equation 2}
         \mi{fit}(\sigma(F)\cap \op(x)\cap \op(y))=\mi{fit}(\sigma(F)\cap \op(x\we y))=F\wF \op(x \we y),
    \end{equation}
   As $\kF$ contains all open sublocales,
    \begin{equation}\label{sigma equation 3}
         F\wF \op(x\we y)=F\wF \op(x)\cap \op(y)=F\wF\op(x)\wF\op(y).
    \end{equation}   
    By combining \ref{sigma equation 2} and \ref{sigma equation 3}, we obtain \ref{sigma equation 1} as desired.\qedhere
    
\end{enumerate} 
\end{proof}

We call $\Delta(\kF)$ the subcolocale $\ca{S}(\sigma[\kF])$. For a complete lattice $C$ and a collection $X\se C$ we call $\ca{J}(X)$ the closure of $X$ under arbitrary joins. If $C$ is a coframe, we call $\ca{S}(X)$ the smallest subcolocale containing $X$.
\begin{lemma}\label{l: general formula for mathcalS}
    Let $\mk{X}\se \sll$ be any subset. Then $\ca{S}(\mk{X})$ is
    \[
    \ca{J}(\{X\cap \op(a)\cap \cl(b)\mid a,b\in L,X\in \mk{X}\}).
    \]
\end{lemma}
\begin{proof}
 By Proposition \ref{p: subcolocales of SL characterization}, any subcolocale $\mk{S}\se \sll$ containing $\mk{X}$ must also contain the collection in the claim. To show the desired result, then, it suffices to show that this is a subcolocale. We use the characterization in \ref{p: subcolocales of SL characterization}. Closure under joins is clear, and stability under $-\cap \op(x)$ and $-\cap \cl(x)$ follows from linearity of open and closed sublocales.
\end{proof}

\begin{lemma}\label{l: delta concretely}
    Let $\kF\se \So(L)$ be a proper subcolocale. Then
    \[
   \Delta(\kF)=\{\bve_i \sigma(F_i)\cap \cl(x_i):F_i\in \kF,x_i\in L\}.
    \]
\end{lemma}
\begin{proof}
By item \ref{sigma basic 3} of Lemma \ref{l: sigma basic}, $\sigma(F)\cap \op(x)\in \sigma[\kF]$ for every $F\in \kF$ and $x\in L$. Then, by Lemma \ref{l: general formula for mathcalS}, $\ca{S}(\sigma[\kF])$ is as desired.
\end{proof}

\begin{proposition}\label{l: proper is fixpoint}
   If $\kF\se \So(L)$ is a proper subcolocale then $\mi{fit}[\Delta(\kF)]=\kF$.
\end{proposition}
\begin{proof}
   The inclusion $\kF\se \mi{fit}[\Delta(\kF)]$ holds by item \ref{sigma basic 1} of Lemma \ref{l: sigma basic}. Let us show the reverse inclusion. We notice that the map $\mi{fit}:\Delta(\kF)\to \So(L)$ preserves all joins, as joins in $\Delta(\kF)$ are computed as in $\sll$. Then, it suffices to prove the claim for basic elements $\sigma(G)\cap \cl(x)$. We now claim that $\mi{fit}(\sigma (G)\cap \cl(x))\in \kF$. Note that, by Lemma \ref{l: intersection and fitting}, 
   \[
   \mi{fit}(\sigma (G)\cap \cl(x))=\mi{fit}(\mi{fit}(\sigma(G))\cap \cl(x)), 
   \]
   and $\mi{fit}(\sigma(G))=G$ by item \ref{sigma basic 1} of Lemma \ref{l: sigma basic}. We have then shown 
   \[
   \mi{fit}(\sigma (G)\cap \cl(x))=\mi{fit} (G\cap \cl(x)).
   \]
   By Proposition \ref{p: subcolocales of SoL characterization}, this is in $\kF$.
\end{proof}
\begin{corollary}
    A subcolocale $\kF\se \So(L)$ is proper if and only if it is of the form $\mi{fit}[\kD]$ for some codense subcolocale $\kD\se \sll$.
\end{corollary}
\begin{proof}
If $\kF$ is proper then $\mi{fit}[\Delta(\kF)]=\kF$ by Proposition \ref{l: proper is fixpoint}. If $\kD\se \sll$ is codense, $\mi{fit}[\kD]$ is proper by Lemma \ref{l: from denseSL to properdenseSoL}. 
\end{proof}

\begin{remark}
Note that one could define $\sigma(F)=\bigcap\{\cl(x)\vee \op(y)\mid x\leq_F y\}$ even when $\kF$ is not proper, but in that case $\sigma(F)\se \cl(x)\vee \op(y)$ does not necessarily imply $x\leq_F y$. Item \ref{sigma basic 1} of Lemma \ref{l: sigma basic}, stating $\mi{fit}(\sigma(F))=F$ for $F\in \kF$, relies on this direction of the implication. This is why one cannot extend the definition of the adjoint $\Delta$ to non-proper subcolocales using this more general definition of $\sigma$. If $F$ is not proper, we cannot prove in a similar way that $\mi{fit}(\sigma(F))=F$, and so our proof of the $\kF\se \mi{fit}[\Delta(\kF)]$ half of the adjointness condition does not go through.   
\end{remark}

\subsection{Essential subcolocales of \texorpdfstring{$\sll$}{SL}}

We look at what subcolocales of $\mf{S}(L)$ are $\Delta(\kF)$ for some proper $\kF\se \So(L)$. For a codense subcolocale $\kD\se \sll$ we call an element \emph{saturated} if it is of the form $\bwe^{\kD}_i \op(x_i)=\nu_{\kD}(\bca_i \op(x_i))$ for some family $x_i\in L$. We call $\mf{Sat}(\kD)$ their ordered collection, and note that $\mf{Sat}(\kD)\se \kD$ is a subcoframe inclusion. Let us define a codense subcolocale $\kD\se \sll$ to be \emph{essential} if it is $\ca{S}(\mf{Sat}(\kD))$. 

\begin{lemma}\label{l: smallest containing the sat}
  For a codense subcolocale $\kD\se \sll$
   \[
   \ca{S}(\mf{Sat}(\kD))=\ca{J}(\{F\cap \cl(z)\mid F\in \mf{Sat}(\kD),z\in L\}).
   \]
\end{lemma}
\begin{proof}
    This follows from the characterization in \ref{l: general formula for mathcalS}.
\end{proof}
We recall that in a category an \emph{essential extension} is a monomorphism $j:A\to B$ such that whenever $m\circ f:A\to C$ is a monomorphism $m$, too, is a monomorphism. An essential extension $n:L\to N$ is \emph{maximal} if for every essential extension $m:L\to M$ there is a unique morphism $f:M\to N$ with $f\circ m=n$. Essential extensions for frames are studied in \cite{ball18}. We now justify the terminology and show that a subcolocale $\kD\se \sll$ is essential if and only if $\mf{Sat}(\kD)\se \kD$ is an essential extension in $\bd{CoFrm}$.

\begin{theorem}
    Let $L$ be a frame and $\mk{D}\se \sll$ a codense subcolocale. Then, $\mk{D}$ is essential if and only if $\mf{Sat}(\mk{D})\se\mk{D}$ is essential in $\bd{CoFrm}$.
\end{theorem}
\begin{proof}
    Suppose that $\kD\se \sll$ is an essential subcolocale. Suppose, now, that there is a coframe map $f:\kD\to C$ such that it is injective when restricted to $\mf{Sat}(\kD)$. The elements of the form $\op(x)\ve \cl(y)$ meet-generate $\kD$, by Lemma \ref{l: very easy}, and as it is essential the elements of the form $F\cap \cl(z)$ with $F\in \mf{Sat}(\kD)$ join-generate it, by Lemma \ref{l: smallest containing the sat}. Then, to show injectivity it suffices to show that $F\cap \cl(z)\nsubseteq \op(x)\vee \cl(y)$ implies that $f(F\cap \cl(z))\nleq f(\op(x)\vee \cl(y))$. Our assumption means that $F\cap \op(y)\nsubseteq \op(x \ve z)$. By assumption on $f$, we have that $f(F\cap \op(y))\nleq f(\op(x \ve z))$, and as $f$ is a coframe map this also implies that $f(F)\we f(\op(y))\nleq f(\op(x))\vee f(\op( z))$. As $f$ also preserves complements, this implies $f(F)\we f(\cl(z))\nleq f(\op(x))\vee f(\cl(y))$, and again as $f$ preserves the lattice operations, $f(F\cap \cl(z))\nleq f(\op(x)\vee\cl(y))$ as desired. If $\mf{Sat}(\kD)\se \kD$ is essential in $\bd{CoFrm}$, consider the coframe quotient given by the subcolocale $\ca{S}(\mf{Sat}(\kD))$. This is clearly injective when restricted to $\mf{Sat}(\kD)$, and by essentiality it is also injective on all of $\kD$. But then it is a bijective coframe map, hence an isomorphism, so $\ca{S}(\mf{Sat}(\kD))=\kD$.
\end{proof}

We show some concrete examples of essential subcolocales.

\begin{lemma}\label{l: bp are T0}
    Let $L$ be a frame. Then $\bl(p)=\cl(p)\cap \mi{fit}(\bl(p))$ for every $p\in \pt(L)$. 
\end{lemma}
\begin{proof}
    We show $\cl(p)\cap \mi{fit}(\bl(p))\se \bl(p)$. Suppose, then, that $x\in \cl(p)\cap \mi{fit}(\bl(p))$ and $x\neq 1$. Since $x\in \cl(p)$, $p\leq x$. Since $x\in \mi{fit}(\bl(p))$, whenever $p\in \op(y)$ then $x\in \op(y)$. As $p$ is prime, $x\to p=p$ if and only if $x\nleq p$, and so our condition means $y\nleq p$ implies $y\to x=x$. As $x\neq 1$, $x\to x\neq x$, and so $x\leq p$. Then $x=p\in \bl(p)$, as desired.
\end{proof}

\begin{proposition}\label{p: concrete examples}
    For a frame $L$, the following are essential subcolocales of $\sll$.
    \begin{enumerate}
        \item The subcolocale $\mf{S}_{sp}(L)$ of spatial sublocales;
        \item The subcolocale $\mf{S}_b(L)$ of joins of complemented sublocales.
    \end{enumerate}
\end{proposition}
\begin{proof}
    For the first item, we recall that $\mf{S}_{sp}(L)$ is join-generated by $\{\bl(p)\mid p\in \pt(L)\}$. Then, it suffices to show that every $\bl(p)$ is $F\cap \cl(p)$ for some $F\in \mf{Sat}(\mf{S}_{sp}(L))$. We let $\nu_{sp}:\sll\to \sll$ be the conucleus associated with $\mf{S}_{sp}(L)\se \sll$. By \ref{o:wF}, and using Lemma \ref{l: bp are T0},
    \[
    \nu_{sp}(\mi{fit}(\bl(p))\cap \cl(p)=\nu_{sp}(\mi{fit}(\bl(p))\cap \cl(p))=\bl(p).
    \]
    By definition of saturated element, $\nu_{sp}(\mi{fit}(\bl(p)))\in \mf{Sat}(\mf{S}_{sp}(L))$, and so the desired claim holds. For the second item, we note that all elements of $\Sb(L)$ are joins of elements of the form $\op(x)\cap \cl(y)$, and $\op(x)\in \mf{Sat}(\Sb(L))$ for every $x\in L$. 
\end{proof}

We now want to characterize essential sublocales as those codense subcolocales of the form $\Delta(\kF)$ for some proper subcolocale $\kF\se \So(L)$.

 \begin{lemma}\label{l: sigma is nucleus}
     Let $\kD\se \sll$ be a codense sublocale. Then, $\nu_{\kD}(\mi{fit}(D))=\sigma(\mi{fit}(D))$ for all $D\in \kD$.
 \end{lemma}
\begin{proof}
    Let $D\in \kD$. We use the characterization in \ref{c: charsigma}. We have to show 
    \[
    \mi{fit}(\nu_{\kD}(\mi{fit}(D))\cap \op(x))=\mi{fit}(D)\wF\op(x)
    \]
    for all $x\in L$. By Lemma \ref{o:wF}, the left-hand side is $\mi{fit}(\nu_{\kD}(\mi{fit}(D)\cap \op(x)))$. By item \ref{from dense SL item 3} of Proposition \ref{p: from denseSL to denseSoL}, this equals the right-hand side.
\end{proof}

\begin{lemma}\label{l:fitsubcofrm}\label{l:subcofrm}
Let $\kD\se \sll$ be a codense subcolocale.
    \begin{enumerate}
        \item \label{fitsubcofrm 1}The maps $\sigma:\mi{fit}[\kD]\lra \kD:\mi{fit}$ are order adjoints, with $\mi{fit}\dashv \sigma$;
        \item \label{fitsubcofrm 2}The map $\sigma:\mi{fit}[\kD]\to \kD$ is a subcoframe embedding;
        \item \label{fitsubcofrm 3}$\sigma[\kF]=\mf{Sat}(\kD)$.
    \end{enumerate}
\end{lemma}

\begin{proof}
Let $\kD\se \sll$ be a dense subcolocale.
\begin{enumerate}
    \item For every $D\in \kD$ and $F\in \mi{fit}[\kD]$, the following are equivalent.
    \begin{prooftree}
        \AxiomC{$\mi{fit}(D)\se F$}
        \UnaryInfC{$D\se F$}
        \UnaryInfC{$D\se \nu_{\kD}(F)$.}
    \end{prooftree}
As $\nu_{\kD}(F)=\sigma(F)$ for all $F\in \mi{fit}[\kD]$, by Lemma \ref{l: sigma is nucleus}, the result follows.

    \item As we have just shown, $\sigma$ is a right adjoint, and so it preserves all meets. For binary joins, it suffices to show it preserves joins of the form $\op(x)\ve \op(y)$, but indeed $\sigma(\op(x\vee y))=\op(x\vee y)$ by item \ref{sigma basic 2} of Lemma \ref{l: sigma basic}. Injectivity of $\sigma$ follows from its left adjoint being surjective. 
    \item By item \ref{sigma basic 2}, of \ref{l: sigma basic}, $\sigma(\op(x))=\op(x)$ for all $x\in L$. Additionally, we have just shown that $\sigma:\mi{fit}[\kD]\to \kD$ is a coframe map. Hence, for all collections $x_i\in L$, $\sigma(\bwe^{\kF}_i \op(x_i))$ is $\nu_{\kD}(\bca_i \op(x_i))$, and so it is saturated. Since this holds for all families $x_i\in L$, all saturated elements of $\kD$ are of this form.\qedhere
\end{enumerate}

\end{proof}

\begin{lemma}\label{l: the main order adjunction}
    For every codense subcolocale $\kD\se \sll$, and every proper subcolocale $\kF\se \So(L)$, $\Delta(\kF)\se \kD$ if and only if $\kF\se \mi{fit}[\kD]$.
\end{lemma}
\begin{proof}
   If $\Delta(\kF)\se \kD$, then $\mi{fit}[\Delta(\kF)]\se \mi{fit}[\kD]$, but the left-hand side is $\kF$, by Proposition \ref{l: proper is fixpoint}. Suppose that $\kF\se \mi{fit}[\kD]$. By definition of $\Delta(\kF)$, it suffices to show $\sigma(F)\in \kD$, for all $F\in \kF$. By our assumption, every such $F$ is $\mi{fit}(D)$ for some $D\in \mi{fit}(D)$, and, as  $\sigma(\mi{fit}(D))=\nu_{\kD}(\mi{fit}(D))$ by Lemma \ref{l: sigma is nucleus}, this is indeed in $\kD$.
\end{proof}
\begin{proposition}\label{p: delta iff fixpoint}
    Let $L$ be a frame. A codense subcolocale $\kD\se \sll$ is essential if and only if $\kD\se \Delta(\mi{fit}[\kD])$.
\end{proposition}
\begin{proof}
By definition, $\Delta(\mi{fit}[\kD])=\ca{S}(\sigma[\mi{fit}[\kD]])$. By Lemma \ref{l:subcofrm}, $\sigma[\mi{fit}[\kD]]=\mf{Sat}(\kD)$. Then, it suffices to show that $\kD$ is essential if and only if $\kD\se \ca{S}(\mf{Sat}(\kD))$, but this just the definition of essentiality.
\end{proof}

We have obtained the main theorem. We call $\ca{CD}_{ess}(\sll)$ the collection of essential codense subcolocales, and $\ca{PC}(\So(L))$ the collection of proper subcolocales of $\So(L)$.
\begin{theorem}\label{t: main theorem}
   There is an order adjunction $\Delta:\ca{PC}(\So(L))\lra \ca{CD}(\sll):\mi{fit}[-]$ with $\Delta \dashv \mi{fit}[-]$, and which maximally restricts to an isomorphism $\ca{PC}(\So(L))\cong \ca{CD}_{ess}(\sll)$.
\end{theorem}
\begin{proof}
    This follows from Lemmas \ref{l: the main order adjunction} and \ref{l: proper is fixpoint}.
\end{proof}

We want to use the characterization in \ref{p: delta iff fixpoint} to provide an example of a subcolocale of $\sll$ which is not essential.

\begin{example}
As $\Sb(L)$ is essential, by Proposition \ref{p: concrete examples}, $\Delta(\mi{fit}[\Sb(L)])=\Sb(L)$ by Proposition \ref{p: delta iff fixpoint}. By Lemma \ref{l: SbL and SeeL have the same fitting}, $\mi{fit}[\Sb(L)]=\mi{fit}[\See(L)]$. Thus, if for some frame we had $\See(L)\nsubseteq\Sb(L)$, this would imply $\See(L)\nsubseteq \Delta(\mi{fit}[\See(L)])$, giving the desired counterexample as $\See(L)$ would not be essential by Proposition \ref{p: delta iff fixpoint}. For sublocales witnessing $\See(L)\nsubseteq \Sb(L)$, once again we refer to Example 5.12 in \cite{bezhanishvili2025mckinseytarskialgebrasraneyextensions} and Example \ref{e: igor example}.
\end{example}

\section{The categories \texorpdfstring{$\bd{Raney}$}{Raney} and \texorpdfstring{$\bd{SZDBF}$}{SZDBF}}

\subsection{Objects}
 We say that a strictly zero-dimensional biframe $(L,\kD)$ is \emph{essential} if $\kD\se \sll$ is an essential subcolocale. We say that a Raney extension $(L,\mk{F})$ is \emph{proper} if $\mk{F}\se \So(L)$ is a proper subcolocale. We will use the results of the previous section to establish a bijection between proper Raney extensions and essential strictly zero-dimensional biframes. For a proper Raney extension $(L,\kF)$ and for a strictly zero-dimensional biframe $(M,\kD)$ we define:
 \begin{align*}
    \Delta(L,\kF)= (L,\Delta(\kF)), && \mi{fit}(M,\kD)=(M,\mi{fit}[\kD]).
 \end{align*}
 Theorem \ref{t: main theorem}, then, amounts to the following.
 \begin{theorem}
  The assignments $\mi{fit}$ and $\Delta$ are mutually inverse bijections between proper Raney extensions and essential strictly zero-dimensional biframes.
 \end{theorem}

\subsection{Morphisms}
The assignment $\mi{fit}$ can be easily extended to morphisms. By Lemma \ref{l:subcofrm}, there is a subcoframe inclusion $\sigma:\mi{fit}[\kD]\to \kD$ for every codense subcolocale $\kD$. Then, every morphism $f:(L,\kD)\to (M,\mk{E})$ determines a coframe map $\mi{fit}(f):\mi{fit}[\kD]\to \mi{fit}[\mk{E}]$, which further restricts to the open sublocales to yield a frame map isomorphic to $f:L\to M$. 

\begin{proposition}
    There is a functor $\mi{fit}:\bd{SZDBF}\to \bd{Raney}$, whose essential image consists of the proper Raney extensions.
\end{proposition}

On the other hand, the assignment $(L,\kF)\mapsto \Delta(L,\kF)$ cannot be extended to morphisms in a similar fashion. We will show that there are Raney morphisms 
\[
f:(L,\kF)\to (M,\mk{G})
\]
such that the frame map $f:L\to M$ does not lift to a map 
\[
f:(L,\Delta(\kF))\to (M,\Delta(\mk{G})).
\]

\begin{lemma}\label{l: delta fit of SbL is SbL}
$\Delta(L,\mi{fit}[\Sb(L)])=(L,\Sb(L))$ for all frames $L$.
\end{lemma}
\begin{proof}
    The strictly zero-dimensional biframe $(L,\Sb(L))$ is essential, by Proposition \ref{p: concrete examples}. By Proposition \ref{p: delta iff fixpoint}, then, $\Delta(\mi{fit}[\Sb(L)])=\Sb(L)$.
\end{proof}

\begin{proposition}\label{p: exact but not smooth}
    For a frame $L$, and a sublocale $S\se L$:
    \begin{enumerate}
        \item $S$ is smooth if and only if it lifts to a map $f:(L,\Sb(L))\to (S,\Sb(S))$ of strictly zero-dimensional biframes.
        \item $S$ is exact if and only if it lifts to a map $f:(L,\mi{fit}[\Sb(L)])\to(S,\mi{fit}[\Sb(L)])$ of Raney extensions.
    \end{enumerate}
\end{proposition}
\begin{proof}
We prove the two items in turn.
\begin{enumerate}
    \item This follows immediately from both Lemma 3.39 in \cite{manuell15} and Corollary 4.2 of \cite{arrieta22}.
    \item In \cite{jakl25} it is proven that there is an isomorphism $\mi{fit}[\Sb(L)]\cong \Sc(L)^{op}$. This is also an isomorphism $(L,\mi{fit}[\Sb(L)])\cong (L,\opp{\Sc(L)})$ of Raney extensions. In Proposition 6.6 of \cite{suarez25} the frame morphisms $f:L \to M$ that lift to Raney extensions are characterized as the exact maps. Finally, in Proposition 7.14 of \cite{suarez25} it is shown that surjections that preserve all exact meets are exact maps. \qedhere
\end{enumerate}
\end{proof}

We are ready to give the desired counterexample.

\begin{example}
   Suppose there is a frame $L$ with a sublocale $S\se L$ which is exact but not smooth. Proposition \ref{p: exact but not smooth} and Lemma \ref{l: delta fit of SbL is SbL} imply that the corresponding frame surjection $s:L\to S$ lifts to a Raney morphism $f:(L,\mi{fit}[\Sb(L)])\to (S,\mi{fit}[\Sb(S)])$ which does not in turn lift to a morphism $f:\Delta(L,\mi{fit}[\Sb(L)])\to \Delta(S,\mi{fit}[\Sb(S)])$ in $\bd{SZDBF}$. Therefore, once again a counterexample is provided by both 5.12 of \cite{bezhanishvili2025mckinseytarskialgebrasraneyextensions} and Example \ref{e: igor example}.
\end{example}

\printbibliography
\end{document}